\documentclass[a4paper,12pt]{article}
\usepackage[english]{babel}
\usepackage{graphicx}
\usepackage{amssymb}
\usepackage{amsthm}
\usepackage{setspace}

\usepackage{amsmath,amsthm,amscd,amsfonts,mathrsfs,amssymb}
\usepackage{graphicx,enumerate}
\usepackage{txfonts} 
\usepackage[all]{xy}
\usepackage{fullpage} 
\usepackage{color}
\usepackage{makeidx}
\usepackage{alltt}

\newtheorem{theorem}{Theorem}[section]
\newtheorem{lemma}[theorem]{Lemma}
\newtheorem{conjecture}[theorem]{Conjecture}

\newtheorem{proposition}[theorem]{Proposition}
\theoremstyle{definition}
\newtheorem{remark}[theorem]{Remark}
\theoremstyle{definition}

\newcommand{\dd}{\mbox{\;\rm d}}
\newcommand{\uhalf}{\mathfrak{h}}

\newcommand{\ggl}[2]{\text{GL}(#1,\mathbb{#2})}
\newcommand{\ssl}[2]{\text{SL}(#1,\mathbb{#2})}
\newcommand{\pgl}[2]{\text{PGL}(#1,\mathbb{#2})}
\newcommand{\ssll}[2]{\mathfrak{sl}(#1,\mathbb{#2})}

\newcommand{\oo}[2]{\text{O}(#1,\mathbb{#2})}
\newcommand{\ssu}[1]{\text{SU}(#1)}

\newcommand{\rankin}{L(1,\phi_j,{\rm Ad})}
\newcommand{\rankinv}{L(1,\varphi,\text{Ad})}
\newcommand{\norm}[1]{\parallel{#1}\parallel}
\newcommand{\tth}{^\text{th}}

\newcommand{\weightj}{\frac{h_T(\nu^{(j)})}{\rankin}}
\newcommand{\error}{\mathcal{O}\left((T^2P^{1/2}+T^3P^\theta+P^{5/3})T^{-5+\epsilon}P^\epsilon\right)}

\begin{document}
\title{Weighted Sato-Tate Vertical Distribution of the Satake Parameter of Maass Forms on PGL(N)}
\author{Fan Zhou}
\date{\today}
\maketitle

\begin{abstract}
We formulate a conjectured orthogonality relation between the Fourier coefficients of Maass forms on PGL($N$). Based on the works of Goldfeld-Kontorovich and Blomer for N=3, and on our conjecture for N$\geq$4, we prove a weighted vertical equidistribution theorem (with respect to the generalized Sato-Tate measure) for the Satake parameter of Maass forms at a finite prime. For N=3, the rate of convergence for the equidistribution theorem is obtained.
\\
\textbf{Mathematics Subject Classiﬁcation:} Primary 11F55; Secondary 11F72; 11F30.
\end{abstract}

\section{Introduction}
Let $$\varphi(z)=\sum\limits_{n=1}^\infty a_\varphi(n)e^{2\pi i nz}$$ be a holomorphic modular form of weight $k$ for the modular group $\ssl{2}{Z}$. We assume that $\varphi$ is a Hecke eigenfunction with normalization $a_\varphi(1)=1$. The Ramanujan conjecture states $$\left|\frac{a_\varphi(p)}{p^{\frac{k-1}{2}}}\right|\leq 2$$ for any prime number $p$. 
It was proved by Deligne in \cite{weil} as a consequence of his proof of the Weil conjectures. We define a measure on $\mathbb{R}$
$$\dd\mu_{\infty}(x)=\begin{cases}
\frac{1}{\pi}\sqrt{1-\frac{x^2}{4}}\dd x,& \text{when } |x|\leq 2,\\
0,&\text{otherwise,}
\end{cases}$$
called the Sato-Tate measure for GL(2), or the semi-circle measure. The Sato-Tate conjecture is a more refined statement about the statistics of the Hecke eigenvalues, stating that 
if $\varphi$ is a non-CM holomorphic modular form of weight $k\geq 2$, then
${a_\varphi(p)}/{p^{\frac{k-1}{2}}}$ is an equidistributed sequence as $p\to \infty$ with respect to the Sato-Tate measure $\mu_{\infty}$. More precisely the Sato-Tate conjecture predicts that 
$$\lim\limits_{T\to\infty}\frac{\sum\limits\limits_{p\leq T}f\left(\frac{a_\varphi(p)}{p^{\frac{k-1}{2}}}\right)}{\sum\limits\limits_{p\leq T}1}=\int_\mathbb{R}f \dd\mu_{\infty}$$
 for any continuous test function $f:\mathbb{R}\to\mathbb{R}$. In recent years many cases of this conjecture have been solved in 
\cite{bgg} and \cite{bght}.

Considering this problem from the vertical perspective, we can fix the prime number $p$ and investigate the distribution of ${a_\varphi(p)}/{p^{\frac{k-1}{2}}}$ as $\varphi$ varies over different automorphic forms. 
In \cite{sarnak}, it was proved that 
$a_\varphi(p)$ is equidistributed with respect to the $p$-adic Plancherel measure
$$\dd\mu_p(x)= \frac{p+1}{(p^{1/2}+p^{-1/2})^2-x^2} \dd\mu_{\infty}(x)$$
 as $\varphi$ runs over all Hecke-Maass cusp forms for the group $\ssl{2}{Z}$. An effective version of \cite{sarnak} appeared in \cite{lw}.
From the same perspective of fixing $p$ and varying $\varphi$, \cite{cdf} and \cite{serre} proved similar equidistribution theorems for holomorphic modular forms, which also involve the Plancherel measure. 
Very recently \cite{shin} gave a highbrow generalization of \cite{sarnak}, \cite{serre} et al.

It is understandable that by fixing a prime number $p$ instead of an automorphic form $\varphi$ we get the $p$-adic Plancherel measure instead of the Sato-Tate measure. Strikingly, if we give each Hecke eigenvalue $a_\varphi(p)$ a weight 
$$\frac{1}{\underset{s=1}{\text{\rm Res }}L(s,\varphi\times \tilde{\varphi})}=\frac{1}{L(1,\varphi,\text{Ad})}$$ and do the same statistics with fixed $p$ and varying $\varphi$, the same Sato-Tate measure appears once again, instead of the $p$-adic Plancherel measure.  
More interestingly, neither the weight $1/\rankinv$ nor the Sato-Tate measure depends on the choice of the prime number $p$.
In \cite{brugg} it was essentially proved that 
$$\lim\limits_{T\to \infty}
\frac{\sum\limits_{\lambda_\varphi\leq T}\frac{f(a_\varphi(p))}{\rankinv}}
{\sum\limits_{\lambda_\varphi\leq T}\frac{1}{\rankinv}}
=\int_{\mathbb{R}}f \dd\mu_\infty
$$
for any continuous test function $f:\mathbb{R}\to\mathbb{R}$, where $\varphi$ runs over all Hecke-Maass forms for $\ssl{2}{Z}$ and $\lambda_\varphi$ is the Laplace eigenvalue of $\varphi$. Later \cite{gmr} and \cite{li}  proved similar theorems for holomorphic modular forms. The weight ${1}/{\rankinv}$ appears naturally in the Petersson and Kuznetsov trace formulae and that is essential to the proofs.

We generalize theorems of such type to a family of cuspidal automorphic representations of
$\pgl{N}{A}$.
The theory of Maass forms for $\ssl{N}{Z}$ ($N\geq 3$) has been studied since the 1980s. The definitions and results are summarized in \cite{g}. The cuspidal part of $\mathcal{L}^2(\ssl{N}{Z}\setminus\ggl{N}{R}/\oo{N}{R}\cdot\mathbb{R}^\times)$ has a discrete spectrum $\phi_1, \phi_2,...$
with $0<\lambda_1\leq \lambda_2\leq ...$ and $\Delta \phi_j=\lambda_j\phi_j$, where $\Delta$ is the Laplace operator and each $\phi_j$ is a Hecke eigenfunction. 
After adelic lifting, each $\phi_j$ corresponds to an irreducible
un-ramified automorphic representation 
$\pi_j$ 
of $\pgl{N}{A}$.
The global representation $\pi_j$ factorizes into local representations $\otimes_{v\leq \infty}\pi_{j,v}$. 
Each Maass form $\phi_j$ has a spectral parameter $\nu^{(j)}=(\nu^{(j)}_1,...,\nu^{(j)}_{N-1})\in\mathbb{C}^{N-1}$, which determines $\pi_{j,\infty}$. Each $\phi_j$ has Fourier coefficients $A_j(m_1,...,m_{N-1})$ for integers $m_1,...,m_{N-1}$ with normalization $A_j(1,...,1)=1$.

For a finite prime $p$, we have $\pi_{j,p}$ an un-ramified principal series representation of $\text{PGL}(N,\mathbb{Q}_p)$. 
Denote the standard maximal torus of $\ssl{N}{C}$ by $T$, the
Weyl group by $W$, and the standard maximal torus of $\ssu{N}\subset\ssl{N}{C}$ by $T_0$.
The Satake isomorphism sends each $\pi_{j,p}$ to a point $X_j(p)$ in $T/W$, which is called the Satake parameter of $\pi_{j,p}$. The generalized Ramanujan conjecture predicts that $\pi_{j,p}$ is tempered and, equivalently, $X_j(p)$ lies in $T_0/W$, which is a proper subset of $T/W$. We define the generalized Sato-Tate measure on $T_0/W$ by pushforwarding the normalized Haar measure of $\ssu{N}$ to $T_0/W$ that sends an element to its conjugacy class.
Denote the Sato-Tate measure on $T_0/W$ by ${\rm d} x$. Whereas in GL(2)  the Hecke eigenvalue at $p$ is enough to characterize the local factor at $p$, it is false when we move to higher dimensions. We shall investigate the distribution of the Satake parameters $X_j(p)\in T/W$ instead of the Hecke eigenvalues, as in \cite{shin}.

Inspired by previous work on the Kuznetsov trace formula and the Petersson trace formula
such as \cite{blomer}, \cite{brugg}, \cite{gk}, \cite{gmr}, \cite{lw}, \cite{li}, \cite{liknightly}, it is natural
to formulate the following conjecture.

\begin{conjecture}[\textbf{Orthogonality relation}]\label{mainconj}
For each $j=1,2,...$, let $A_j(m_1,...,m_{N-1})$ denote the $(m_1,...,m_{N-1})\tth$ Fourier
coefficient of a Maass form $\phi_j$ for $\ssl{N}{Z}$ with $N\geq2$. 
For each $T\gg 1$, let $h_T$ be any bounded non-negative test function on the spectral parameters which, roughly speaking, essentially counts Maass forms with spectral parameters $|\nu^{(j)}_i|\leq T$. We will elaborate on the definition of $h_T$ later in Conjecture \ref{or1}.
We conjecture that  the following orthogonality relation holds:
\begin{equation}\label{orthogonality}
\lim\limits_{T\to \infty}\frac{\sum\limits_{j=1}^\infty A_j(m_1,,...,m_{N-1})\overline{A_j(n_1,...,n_{N-1})}\weightj}{\sum\limits_{j=1}^\infty \weightj}=\begin{cases}
1, & \text{if }m_i=n_i\text{ for all }i,\\
0, & \text{otherwise}. 
\end{cases}
\end{equation}
\end{conjecture}

Conjecture \ref{mainconj} was proved for $N=2$  in \cite{brugg}, and for $N=3$
in \cite{gk} and \cite{blomer}, where the following stronger result with error term (Theorem \ref{gk}) is obtained. It states 
\begin{multline}\label{gl3or}
\sum\limits_{j=1}^\infty A_j(m_1,m_2)\overline{A_j(n_1,n_2)}\frac{h_{T}(\nu^{(j)})}{\rankin}\\
=\delta_{m_1,n_1}\delta_{m_2,n_2}{\sum\limits_{j=1}^\infty \frac{h_{T}(\nu^{(j)})}{\rankin}}+\mathcal{O} \left((T^2P^{1/2}+T^3P^\theta+P^{5/3})(TP)^\epsilon\right)
\end{multline}
for some $\theta\leq 7/64$ and $P=m_1m_2n_1n_2$.

\begin{theorem}[\textbf{Main theorem}]\label{mainthmCONJ}
Let $\phi_1,\phi_2,...$ be the basis of Maass forms for $\ssl{N}{Z}$. Each $\phi_j$ corresponds to an irreducible un-ramified automorphic representation $\pi_j$ of $\pgl{N}{A}$ with the Satake parameter $X_j(p)\in T/W$ at a finite prime $p$. 
Assume Conjecture \ref{mainconj} if $N\geq 4$.
For any continuous test function $f: T/W\to \mathbb{C}$, we have the equality
\begin{equation}\label{mainequation}
\lim\limits_{T\to \infty}\frac{\sum\limits_{j=1}^\infty f(X_j(p))\weightj}{\sum\limits_{j=1}^\infty \weightj}=\int_{T_0/W}f(x)\dd x.
\end{equation}
\end{theorem}

Our main idea of the proof is to translate the Fourier coefficients $A_j(m_1,...,m_{N-1})$ in Equation \ref{orthogonality} into the characters of finite-dimensional representations of $\ssu{N}$, via the Casselman-Shalika formula. We complete the proof after some computation and an application of the Peter-Weyl theorem (or the Stone-Weierstrass theorem). 

Theorem \ref{mainthmCONJ} essentially proves that the Ramanujan conjecture $X_j(p)\in T_0/W$ holds in average in the vertical sense, i.e., for fixed $p$ and varying $j$. This is because the left side of Equation \ref{mainequation} has $X_j(p)\in T/W$ while the Sato-Tate measure ${\rm d} x$ on the right side of Equation \ref{mainequation} is only supported on $T_0/W$.

\begin{small}
\begin{remark} 
After we submitted the preprint of this paper to the arXiv, we were notified that Theorem \ref{mainthmCONJ} (for the case $N=3$) was independently proved in \cite{bbr}.
\end{remark}
\end{small}

As in \cite{lw}, \cite{ms}, and \cite{shin}, 
we also obtain an effective version of Theorem \ref{mainthmCONJ} for $N=3$, which gives the rate of
convergence, but only for monomial functions.
Its proof is based on the error term in the orthogonality relation (Equation \ref{gl3or}). 

\begin{theorem}[\textbf{Rate of convergence for N=3}]\label{convergi}
Let $\phi_1,\phi_2,...$ be the basis of Maass forms for $\ssl{3}{Z}$.
Let $f: T/W\to \mathbb{C}$ be defined as
$$f\left(\begin{pmatrix}
\alpha_1&&\\
&\alpha_2&\\
&&\alpha_3
\end{pmatrix}\right)=\left(\sum\limits_{i=1}^3\alpha_i\right)^{i_1}\left(\sum\limits_{i=1}^3\overline{\alpha_i}\right)^{i'_1}
\left(\sum\limits_{1\leq i<j\leq 3}\alpha_i\alpha_j\right)^{i_2}\left(\sum\limits_{1\leq i<j\leq 3}\overline{\alpha_i\alpha_j}\right)^{i'_2}$$
for non-negative integers $i_1$, $i'_1$, $i_2$, $i'_2$. 
For $T\gg 1$, let $h_T$ be defined as in Theorem \ref{gk}.
For fixed $\epsilon>0$, we have the asymptotic formula with error term
$$\frac{\sum\limits_{j=1}^\infty f(X_j(p))\frac{h_{T}(\nu^{(j)})}{\rankin}}{\sum\limits_{j=1}^\infty \frac{h_{T}(\nu^{(j)})}{\rankin}}=\int_{T_0/W}f(x)\dd x+\error
$$
as $T\gg 1$, for some $\theta\leq 7/64$ and $P=p^{i_1+i'_1+i_2+i'_2}$.
\end{theorem}

\section{Background on Maass Forms}

Our main reference is \cite{g} for this section. Fix an integer $N\geq 2$. The cuspidal part of $\mathcal{L}^2\left(\ssl{N}{Z}\setminus\ggl{N}{R}/\oo{N}{R}\cdot\mathbb{R}^\times\right)$ has a discrete spectrum $\phi_1, \phi_2,...$ 
with $\Delta \phi_j=\lambda_j\phi_j$ and $0<\lambda_1\leq \lambda_2\leq...$, where $\Delta$ is the Laplace operator and each $\phi_j$ is a Hecke eigenfunction.
Via adelic lifting, each Hecke-Maass form $\phi_j$ corresponds to an irreducible automorphic representation 
$\pi_j=\otimes_{v\leq \infty}\pi_{j,v}$
of $\pgl{N}{A}$. The asymptotic behavior of this discrete spectrum, namely, the Weyl law,  has been studied since Selberg introduced his trace formula.

Let $\uhalf^N$ be the generalized upper half-plane consisting of $z=x\cdot y$, where
$$x=\begin{pmatrix}
1&x_{1,2}&x_{1,3}&\cdots&&x_{1,N}\\
&1&x_{2,3}&\cdots&&x_{2,N}\\
&&\ddots&&&\vdots\\
&&&1&x_{N-2,N-1}&x_{N-2,N}\\
&&&&1&x_{N-1,N}\\
&&&&&1
\end{pmatrix} \;\text{ and }\; y=\begin{pmatrix}
y_1y_2\cdots y_{N-1}&&&&\\
&\ddots&&&\\
&&y_1y_2&&\\
&&&y_1&\\
&&&&1
\end{pmatrix}$$
with $x_{*,*}\in \mathbb{R}$ and $y_*>0$. 
Each Maass form $\phi_j$ is a smooth function in $\mathcal{L}^2(\ssl{N}{Z}\setminus\ggl{N}{R}/\oo{N}{R}\cdot\mathbb{R}^\times)$. By the Iwasawa decomposition $\uhalf^N\simeq
\ggl{N}{R}/\oo{N}{R}\cdot\mathbb{R}^\times,$ we can view $\phi_j$ as a function on $\uhalf^N$ invariant on the left by the action of $\ssl{N}{Z}$.

A Maass form $\phi:\ssl{N}{Z}\setminus\uhalf^N\to\mathbb{C}$ with spectral parameter $\nu=(\nu_1,...,\nu_{N-1})\in\mathbb{C}^{N-1}$ has Fourier-Whittaker expansion $$\phi(z)=\sum\limits_{\gamma\in \text{U}_{N-1}(\mathbb{Z})\setminus\ssl{N-1}{Z}}\sum\limits_{m_1=1}^\infty\cdots \sum\limits_{m_{N-2}=1}^\infty\sum\limits_{m_{N-1}\neq 0}\frac{A(m_1,...,m_{N-1})}{\prod\limits_{k=1}^{N-1}|m_k|^{\frac{k(N-k)}{2}}}
W_{\text{Jacquet}}\left(M\cdot\begin{pmatrix}
\gamma &\\
&1
\end{pmatrix};\nu+\frac{1}{N},\psi_{1,...,1,\frac{m_{N-1}}{|m_{N-1}|}}\right),$$
where $W_\text{Jacquet}$ is Jacquet's Whittaker function,
$(\nu+\frac{1}{N})$ means $(\nu_1+\frac{1}{N},...,\nu_{N-1}+\frac{1}{N})$
and 
$$M=\begin{pmatrix}
m_1...m_{N-2}|m_{N-1}|&&&&\\
&\ddots&&&\\
&&m_1m_2&&\\
&&&m_1&\\
&&&&1
\end{pmatrix}.$$
 We choose to normalize $\phi$ by requiring that $A(1,...,1)=1$.
The number $A(m_1,...,m_{N-1})$ is the $(m_1,...,m_{N-1})\tth$-Fourier coefficient of $\phi$.

We define $$b_{ij}=\begin{cases}
ij, &\text{ if }i+j\leq N,\\
(N-i)(N-j), &\text{ if }i+j>N.\end{cases}$$
For the spectral parameter $\nu=(\nu_1,...,\nu_{N-1})$, we define 
$B_j(\nu)=\sum\limits_{i=1}^{N-1} b_{ij}v_i$.
We define the Langlands parameter $\ell=(\ell_1,...,\ell_{N})\in\mathbb{C}^N$ by
$$\ell_i(\nu)=\begin{cases}
B_{N-1}(\nu),&\text{ if }i=1,\\
B_{N-i}(\nu)-B_{N-i+1}(\nu),&\text{ if }1<i<N,\\
-B_{1}(\nu),&\text{ if }i=N.
\end{cases}$$
A basic fact is that 
$$\lambda(\nu)=\frac{N^3-N}{24}-\frac{1}{2}\sum_{i=1}^N \ell_i^2(\nu)$$ is the Laplace eigenvalue of $\phi$, i.e., $\Delta\phi=\lambda(\nu)\phi$.
For each Maass form, we can find unique $\nu$ to be its spectral parameter such that 
we have $\mathfrak{I}\nu_i\geq 0$ for all $i=1,...,N-1$.

We define the Rankin-Selberg convolution $L$-function for a Hecke-Maass form $\phi$ with its contragredient to be
$$L(s,\phi\times \tilde{\phi})=\zeta(Ns)\sum\limits_{m_1=1}^\infty\cdots\sum\limits_{m_{N-1}=1}^\infty \frac{|A_j(m_1,...,m_{N-1})|^2}{m_1^{{(N-1)}s}\cdots m_{N-1}^s},$$ where $\zeta(s)=\sum\limits_{n=1}^\infty \frac{1}{n^s}$ is the Riemann-Zeta function.
We have 
$$L(s,\phi\times \tilde{\phi})=\zeta(s)L(s,\phi,\text{Ad})\;\text{ and }\;{\underset{s=1}{\text{\rm Res }}L(s,\phi\times \tilde{\phi})}=L(1,\phi,\text{Ad}),$$ where $L(s,\phi,\text{Ad})$ is the $L$-function defined by the adjoint representation of $\ssl{N}{C}$ (the dual group of PGL($N$))
$$\text{Ad}:\ssl{N}{C}\to\ggl{N^2-1}{C}={\rm GL}(\ssll{N}{C}).$$
In the case of GL(2), the adjoint representation of $\ssl{2}{C}$ is the same as the symmetric square representation.

\section{The Satake Parameter and the Sato-Tate Measure}
We give the definitions of the Satake parameters, the Sato-Tate measure and the Ramanujan conjecture for Maass forms on PGL($N$) in this section.
We are happy to refer to \cite{shin} for more general definitions on other groups.

The standard maximal torus of  $\ssl{N}{C}$ is 
$$T=\left\{\begin{pmatrix}
\alpha_{1}&&\\
&\ddots&\\
&&\alpha_{N}
\end{pmatrix}:\alpha_i\in\mathbb{C}^*\text{ for all }i, \prod\limits_{i=1}^N \alpha_i=1\right\}.$$
The group $\ssu{N}$ is the standard maximal compact subgroup of $\ssl{N}{C}$. The standard maximal torus of  $\ssu{N}$ is 
$$T_0=\left\{\begin{pmatrix}
\alpha_{1}&&\\
&\ddots&\\
&&\alpha_{N}
\end{pmatrix}:\alpha_i\in\mathbb{C}^*\text{ and }|\alpha_i|=1\text{ for all }i, \prod\limits_{i=1}^N \alpha_i=1\right\}.$$ 
The Weyl group $W$ is isomorphic to the symmetric group of $N$ elements and acts on $T$ and $T_0$ by permutation of the diagonal entries. The conjugacy classes of $\ssu{N}$ (or $\ssl{N}{C}$) are one-to-one corresponding to elements in $T_0/W$ (or $T/W$). 
The space $T_0/W$ has a natural normalized measure. This measure is the pushforward measure of the normalized Haar measure on $\ssu{N}$ by the map $\ssu{N}\to T_0/W$ sending an element to its conjugacy class. Let us denote this measure on $T_0/W$ by ${\rm d} x$ and we call this measure $\dd x$ on $T_0/W$ the generalized \textbf{Sato-Tate measure}.

For each Hecke-Maass form $\phi_j$, adelic lifting gives a global automorphic representation
$\pi_j=\otimes_{v\leq \infty}\pi_{j,v}$
of $\pgl{N}{A}$. For a finite prime $p$, each $\pi_{j,p}$ is an un-ramified principal series representation of $\text{PGL}(N,\mathbb{Q}_p)$.
By the Satake isomorphism, this un-ramified principal series representation $\pi_{j,p}$ is associated with $N$ nonzero complex numbers $\alpha_{p,1}$, $\alpha_{p,2}$,...,$\alpha_{p,N}$ with $\prod\limits_{i=1}^N\alpha_{p,i}=1$. 
These numbers $\alpha_{p,i}$ determines the representation $\pi_{j,p}$.
We can recover $\pi_{j,p}$ by constructing the space
$$\left\{\text{smooth function }f:\text{PGL}(N,\mathbb{Q}_p)\to\mathbb{C}:f\left(\begin{pmatrix}
t_1&&*\\
&\ddots&\\
&&t_N
\end{pmatrix} g\right)=\left(\prod\limits_{i=1}^N |t_i|_p^{\frac{N+1}{2}-i} \alpha_{p,i}^{\text{ord}_p(t_i)}\right)f(g)\right\}$$ and $\text{PGL}(N,\mathbb{Q}_p)$ acts on $f$ from the right.
The Satake isomorphism sends $\pi_{j,p}$ to a point 
$$X_j(p)=\begin{pmatrix}
\alpha_{p,1}&&\\
&\ddots&\\
&&\alpha_{p,N}
\end{pmatrix}\in T/W.$$
We define this point $X_j(p)$ in $T/W$ as the \textbf{Satake parameter} of this un-ramified principal series representation $\pi_{j,p}$ of $\text{PGL}(N,\mathbb{Q}_p)$.

The generalized Ramanujan conjecture claims that $\pi_{j,p}$ is tempered
when it comes from a Hecke-Maass form $\phi_j$ and, equivalently, the Satake parameter $X_j(p)$ lies in $T_0/W$ which is a proper subspace of $T/W$. More explicitly the Ramanujan conjecture claims that $|\alpha_{p,i}|=1$ for $i=1,2,...,N$. The Ramanujan conjecture has not been proved for Maass forms, even when $N=2$, as of February 2013.

\section{The Root System of Type A}
The Lie group $\ssl{N}{C}$ and its maximal compact subgroup $\ssu{N}$ are associated with the root system of type $A_{N-1}$. We construct the $A_{N-1}$ root system in $$\left\{(x_1,x_2,...,x_N)\in\mathbb{R}^N:\sum\limits_{i=1}^N x_i=0\right\}.$$
Let $\epsilon_i$ be the vector in $\mathbb{R}^N$ with $1$ at the $i^{\text{th}}$ entry and $0$ elsewhere. We have the set of roots $\Phi=\{\epsilon_i-\epsilon_j:i\neq j\}$. 
We pick up the set of positive roots $\Phi^+=\{\epsilon_i-\epsilon_j:i< j\}$. We have $(\epsilon_i-\epsilon_{i+1})$ for $i=1,2,...,N-1$ as the simple roots of $\Phi^+$. 
Denote the zero weight by $\textbf{0}=(0,...,0)$.

Let $\Lambda$ be the set of integral weights, which is $\mathbb{Z}$-module generated by $\left( \epsilon_i-\frac{1}{N}{\sum\limits_{j=1}^N \epsilon_j}\right)$ for $i=1,2,...,N-1$.  Let $C\subset \left\{(x_1,x_2,...,x_N)\in\mathbb{R}^N:\sum\limits_{i=1}^N x_i=0\right\}$ be the Weyl chamber associated with the positive roots in $\Phi^+$. Explicitly we have $$C=\left\{\sum\limits_{i=1}^N a_i \epsilon_i 
:a_1\geq a_2\geq ...\geq a_N, a_i\in \mathbb{R}, \sum\limits_{i=1}^N a_i=0\right\}.$$

For a weight $\mu\in\Lambda\cap C$ we define $V_\mu$ as the highest weight representation of $\mu$. 
It can be a representation of $\ssu{N}$ or $\ssl{N}{C}$, by the basic Lie theory.
Moreover, each irreducible finite-dimensional complex linear representation of $\ssu{N}$ or $\ssl{N}{C}$ is associated with such a highest weight in $\Lambda\cap C$. Let $\chi_\mu$ be the character of this representation. The character $\chi_\mu$ is a well-defined function on conjugacy classes, $T/W$ and $T_0/W$. Formally each $\chi_\mu$ is a finite sum of $e^{\eta}$ for $\eta\in\Lambda$ with non-negative integer coefficients, invariant under the action of the Weyl group $W$.

Let $V_1$ be the representation of the standard defining map $\ssl{N}{C}\hookrightarrow\ggl{N}{C}$. This representation corresponds to the highest weight representation of  
$(\epsilon_1-\frac{1}{N}\sum\limits_{j=1}^{N}\epsilon_j)\in\Lambda\cap C$. 
Its character is $$\chi_{\epsilon_1-\frac{1}{N}\sum\limits_{j=1}^{N}\epsilon_j}                  
\left(\begin{pmatrix}
\alpha_{1}&&\\
&\ddots&\\
&&\alpha_{N}
\end{pmatrix}\right)=\sum\limits_{i=1}^N\alpha_i.$$
Formally we have 
$$\chi_{\epsilon_1-\frac{1}{N}\sum\limits_{j=1}^{N}\epsilon_j}=\sum\limits_{i=1}^N{e^{\epsilon_i-\frac{1}{N}\sum\limits_{j=1}^{N}\epsilon_j}},$$
where $e^{\epsilon_i-\frac{1}{N}\sum\limits_{j=1}^{N}\epsilon_j}$ corresponds to a character of $T$ or $T_0$ 
$$\begin{pmatrix}
\alpha_{1}&&\\
&\ddots&\\
&&\alpha_{N}
\end{pmatrix}\mapsto \alpha_i.$$

Denote the exterior product $\wedge^k V_1$ by $V_k$ for $k=2,...,N-1$ and $V_k$ corresponds to the highest weight representation of $\sum\limits_{i=1}^k\left(\epsilon_i-\frac{1}{N}\sum\limits_{j=1}^{N}\epsilon_j\right)$.
Denote its character $\chi_{\sum\limits_{i=1}^k\left(\epsilon_i-\frac{1}{N}\sum\limits_{j=1}^{N}\epsilon_j\right)}$ by $\chi_k$ for abbreviation. We have the explicit formula
$$\chi_k\left(\begin{pmatrix}
\alpha_{1}&&\\
&\ddots&\\
&&\alpha_{N}
\end{pmatrix}\right)=\chi_{\sum\limits_{i=1}^k\left(\epsilon_i-\frac{1}{N}\sum\limits_{j=1}^{N}\epsilon_j\right)}\left(\begin{pmatrix}
\alpha_{1}&&\\
&\ddots&\\
&&\alpha_{N}
\end{pmatrix}\right)=\sum\limits_{i_1<i_2<...<i_k}      \alpha_{i_1}\alpha_{i_2}...\alpha_{i_k}.$$ 
These $(N-1)$ representations $V_1,...,V_{N-1}$ are the fundamental representations of $\ssu{N}$ and $\ssl{N}{C}$.
It is obvious that $\chi_1$, $\chi_2$,..., $\chi_{N-1}$ are elementary symmetric polynomials on $T/W$ and $T_0/W$.


\section{The Casselman-Shalika Formula and the Fourier Coefficients at $p$}
Let $\Omega_N$ be defined as $\{(l_1,...,l_{N-1})\in\mathbb{Z}^{N-1}:l_1,...,l_{N-1}\geq 0\}.$
We define a bijective map 
$$\aleph:\Omega_N\to\Lambda\cap C$$ by taking 
$$(l_1,...,l_{N-1})\mapsto\sum\limits_{i=1}^{N-1}
\left(\sum\limits_{k=1}^{N-i} l_k\right)
\left( \epsilon_i-\frac{1}{N}{\sum\limits_{j=1}^N \epsilon_j}\right).$$

\begin{proposition}[\textbf{Casselman-Shalika}]\label{CS}
Let $\phi_j$ be a Hecke-Maass form for $\ssl{N}{Z}$ with Fourier coefficients $A_j(\cdot,...,\cdot)$. 
Let $X_j(p)$ be its Satake parameter at a finite prime $p$. We have 
$$A_j(p^{l_1},...,p^{l_{N-1}})=\chi_{\aleph((l_1,...,l_{N-1}))}\left(X_j(p)\right)$$
for $l_1,...,l_{N-1}\geq 0$.
\end{proposition}

\begin{proof}
The Hecke-Maass form $\phi_j$ can be adelically lifted to a cuspidal automorphic form $\Phi_j$ in 
$\mathcal{L}_\text{cusp} ^2\left(\text{Z}(\mathbb{A})\ggl{N}{Q}\setminus\ggl{N}{A}\right)$.
This automorphic form $\Phi_j$ has a unique global Whittaker function $W(*;\Phi_j)$. It has factorization $$W(g;\Phi_j)=\prod\limits_{v\leq \infty}W_v(g_v;\Phi_j).$$ The automorphic form $\Phi_j$ generates an automorphic representation
$\pi_j$ of $\pgl{N}{A}$ which factorizes into $\otimes_{v\leq \infty} \pi_{j,v}$.
With some minor adelic computation, we obtain $$W_p\left(\begin{pmatrix}
p^{l_1+...+l_{N-1}}&&&\\
&\ddots&&\\
&&p^{l_1}&\\
&&&1
\end{pmatrix};\Phi_j\right)=\frac{A_j(p^{l_1},...,p^{l_{N-1}})}{\prod\limits_{k=1}^{N-1}p^{\frac{l_k k(N-k)}{2}}}.$$
The un-ramified principal series $\pi_{j,p}$ of $\text{PGL}(N,\mathbb{Q}_p)$ also has a Whittaker function
$W_p(*;\pi_{j,p})$ and by normalization 
$W_p(1;\pi_{j,p})=1$ we have 
$$W_p\left(\begin{pmatrix}
p^{l_1+...+l_{N-1}}&&&\\
&\ddots&&\\
&&p^{l_1}&\\
&&&1
\end{pmatrix};\pi_{j,p}\right)=\frac{\chi_{\aleph((l_1,...,l_{N-1}))}\left(X_j(p)\right)}{\prod\limits_{k=1}^{N-1}p^{\frac{l_k k(N-k)}{2}}}$$ from \cite{cash}.
By the multiplicity one theorem, we have 
$$W_p(*;\pi_{j,p})=W_p(*;\Phi_j).$$ 
Evaluating the previous equality at $\begin{pmatrix}
p^{l_1+...+l_{N-1}}&&&\\
&\ddots&&\\
&&p^{l_1}&\\
&&&1
\end{pmatrix}$ we prove the theorem.
\end{proof}

\section{The Orthogonality Relation}
Recall that $\phi_1, \phi_2, ...$ are Hecke-Maass forms for $\ssl{N}{Z}$ with increasing Laplace eigenvalues. 
We rewrite Conjecture \ref{mainconj} in detail as Conjecture \ref{or1} and also introduce Conjecture \ref{or2}.

\begin{conjecture}[\textbf{Orthogonality relation}]\label{or1}
For each $T\gg 1$, let $h_T:\{\nu:\mathfrak{R}\ell_i(\nu)\leq 1/2,i=1,...,N\}\to\mathbb{R}$ be any 
non-negative test function satisfying the following conditions:
\begin{itemize}
\item $h_T\asymp 1$ on $$\left\{\nu=(\nu_1,...,\nu_{N-1}):c\leq\mathfrak{I}\nu_i\leq c_iT,
\text{ for }i=1,...,N-1,\text{ and } |\mathfrak{R}\ell_k(\nu)|\leq 1/2 \text{ for }k=1,...,N\right\}$$ for some positive numbers $c, c_1,...,c_{N-1}$;
\item $h_T(\nu)\ll_A |\lambda(\nu/T)|^{-A}$ for all $A>0$.
\end{itemize}
For each Maass form $\phi_j$, let $$\omega_j(T)=\frac{h_T(\nu^{(j)})}{\rankin}$$ be its weight.
We expect
$$\quad\sum\limits_{j=1}^\infty \omega_j(T)\ll_T 1,\quad\;
\sum\limits_{j=1}^\infty A_j(m_1,...,m_{N-1})\overline{A_j(n_1,...,n_{N-1})}\omega_j(T)\ll_T 1
$$ for all positive integers $m_i$ and $n_i$.
We conjecture that  the following orthogonality relation holds:
$$\lim\limits_{T\to \infty}\frac{\sum\limits_{j=1}^\infty A_j(m_1,,...,m_{N-1})\overline{A_j(n_1,...,n_{N-1})}\omega_j(T)}{\sum\limits_{j=1}^\infty \omega_j(T)}=\begin{cases}
1, & \text{if }m_i=n_i\text{ for all }i,\\
0, & \text{otherwise}. 
\end{cases}$$
\end{conjecture}

In actual application, it is often more useful to construct the family of test functions $h_T$ explicitly. One important example of $h_T$ is $$h_T(\nu)=\exp\left(-\frac{\lambda(\nu)}{T^2}\right).$$ Another example of $h_T$ is $$h_T(\nu)=\begin{cases}
1,&\text{ if }\lambda(\nu)\leq T^2\\
0,&\text{ otherwise,}
\end{cases}$$
which corresponds to the classical Weyl's law.

\begin{conjecture}[\textbf{Weak orthogonality relation}]\label{or2}
Under the same assumption for $h_T$ and $\omega_j(T)$ as in Conjecture \ref{or1}, 
we conjecture that the following weak orthogonality relation holds:
$$\lim\limits_{T\to \infty}\frac{\sum\limits_{j=1}^\infty A_j(m_1,...,m_{N-1})\omega_j(T)}{\sum\limits_{j=1}^\infty \omega_j(T)}=\begin{cases}
1, & \text{if }m_1=m_2=...=m_{N-1}=1,\\
0, & \text{otherwise}. 
\end{cases}$$
\end{conjecture}

Obviously Conjecture \ref{or1} implies Conjecture \ref{or2} because of the normalization $A_j(1,...,1)=1$. 
By applying the Casselman-Shalika formula or the Hecke relations, one can prove the inverse is also true. Hence Conjecture \ref{or1} and Conjecture \ref{or2} are equivalent.

We predict that Conjecture \ref{or1} can be derived from the Kuznetsov trace formula. For $N=2$, Proposition 4.1 of \cite{brugg} gives a version of Conjecture \ref{or1} and numerous similar identities are obtained for various cases on GL(2). The works of  \cite{gk} and \cite{blomer} establish versions of Conjecture \ref{or1} for $N=3$.  For $N\geq 4$, this conjecture is still open. We shall emphasize that the weight $$\frac{1}{\rankin}=\frac{1}{L(1,\phi_j,{\text{Ad}})}$$ is crucial for the orthogonality relation to hold. 
We can see this difference most clearly in \cite{lw}. Lemma 3.3 \cite{lw} is un-weighted and there is no orthogonality, whereas Lemma 3.1 \cite{lw} is weighted with ${1}/\rankin$ and orthogonality holds.

There are numerous applications of the orthogonality relation.
The orthogonality relations with error terms for $N=2,3$ are applied  to studying the symmetry types of the low-lying zeroes of families of $L$-functions in \cite{am}, \cite{six}, \cite{guloglu}, and \cite{gk}.
For $N=2$, it is also applied to Sato-Tate distribution of Hecke eigenvalues in \cite{brugg}, \cite{li}, and \cite{liknightly} and to $p$-adic Plancherel distribution of Hecke eigenvalues in \cite{lw}. We extend the application to Sato-Tate distribution further to $N\geq 3$ in Theorem \ref{main1} and \ref{main2}. The following orthogonality relation for $N=3$ is obtained in \cite{gk} and improved to the current version in the appendix of \cite{blomer}.

\begin{theorem}[\textbf{Orthogonality relation for $N=3$}]
\label{gk}
Let $m_1$, $m_2$, $n_1$, $n_2$ be positive integers and let $P=m_1m_2n_1n_2$.
Let $\theta\leq 7/64$ be a bound towards the Ramanujan conjecture on GL(2).
For $T\gg 1$, we define 
$$\omega_j(T)=\frac{h_{T}(\nu^{(j)})}{\rankin}.$$
Here $h_T$ is non-negative, uniformly bounded on 
$\{|\mathfrak{R}\nu_1|\leq 1/2\}\times\{|\mathfrak{R}\nu_2|\leq 1/2\}$, with
$h_T\asymp 1$ on $\{(\nu_1,\nu_2):c\leq \mathfrak{I}\nu_1,\mathfrak{I}\nu_2\leq T, |\mathfrak{R}\nu_1|,|\mathfrak{R}\nu_2|\leq 1/2\}$ for some absolute constant $c>0$, and $h_T(\nu_1,\nu_2)\ll_A((1+|\nu_1|/T)(1+|\nu_2|/T))^{-A}$.
We have
$$\sum\limits_{j=1}^\infty A_j(m_1,m_2)\overline{A_j(n_1,n_2)}\omega_j(T)=\begin{cases}
\sum\limits_{j=1}^\infty \omega_j(T)+
\mathcal{O} \left((T^2P^{1/2}+T^3P^\theta+P^{5/3})(TP)^\epsilon\right), 
&\text{if }\underset{\displaystyle m_2=n_2,}{m_1=n_1}\\
\mathcal{O} \left((T^2P^{1/2}+T^3P^\theta+P^{5/3})(TP)^\epsilon\right), & \text{otherwise}.
\end{cases}$$
Additionally we have the Weyl's law $$\sum\limits_{j=1}^\infty \omega_j(T)\asymp T^{5}.$$
\end{theorem}
\begin{proof}
See the appendix of \cite{blomer}. Be reminded that in the context of \cite{blomer} we have 
$||\phi_j||^2\asymp \rankin$ (Lemma 1 \cite{blomer}) and the quotient 
$||\phi_j||^2/ \rankin$ is a gamma factor that only depends on the spectral parameter $\nu^{(j)}$.
\end{proof}

\section{A Short Proof of the Main Theorem Under the Assumption of the Ramanujan Conjecture}

Let us assume the Ramanujan conjecture at a finite prime $p$ which states that the Satake parameter $X_j(p)$ of a Hecke-Maass form $\phi_j$ has the property $$X_j(p)\in T_0/W.$$ Let $\mathcal{C}(T_0/W)$ be the space of complex-valued continuous functions on $T_0/W$. It is a Banach space under the supremum norm $\norm{f}_\infty=\underset{x\in T_0/W}{\text{sup}} |f(x)|$.
Let us recall that $\Lambda \cap C$ is the set of positive weights for the root system $A_{N-1}$.
Any character $\chi_\mu$ lies in $\mathcal{C}(T_0/W)$ for $\mu\in\Lambda\cap C$.
We define the linear subspace spanned by characters $$\mathcal{B}=\left\{\sum\limits_{\mu\in \Lambda\cap C} a_\mu \chi_\mu: a_\mu\in \mathbb{C}, a_\mu=0\text{ for all but finitely many }\mu\right\}.$$ 

\begin{theorem}[\textbf{Peter-Weyl}]\label{peterweyl}
The subspace $\mathcal{B}$ is dense in $\mathcal{C}(T_0/W)$, under the topology of the supremum norm.
\end{theorem}
\begin{proof}
See p.23 of \cite{bump} and p.134 of \cite{dieck}. 
This is a less known version of Peter-Weyl theorem than its $L^2$-version. 
\end{proof}

\begin{lemma}
Assume Conjecture \ref{or2} if $N\geq 4$.
For any $f\in \mathcal{B}$, we have the equality
$$\lim\limits_{T\to \infty}\frac{\sum\limits_j f(X_j(p))\omega_j(T)}{\sum\limits_j \omega_j(T)}=\int_{T_0/W}f(x)\dd x.$$
\end{lemma}
\begin{proof}
We only need to prove for $f=\chi_\mu$ for all $\mu\in \Lambda\cap C$. 
Recall a corollary of the Schur orthogonality relations
$$\int_{T_0/W}\chi_\mu(x)\dd x=\begin{cases}
1, & \text{if }\mu=\textbf{0},\\
0, & \text{otherwise}. 
\end{cases}$$
If $f=\chi_\textbf{0}\equiv 1$ (constant function), we have
$$\frac{\sum\limits_j f(X_j(p))\omega_j(T)}{\sum\limits_j \omega_j(T)}=1=\int_{T_0/W}f(x)\dd x.$$
If $f=\chi_\mu$ for $\mu\neq \textbf{0}$, we have
\begin{eqnarray*}
\lim\limits_{T\to\infty}\frac{\sum\limits_j f(X_j(p))\omega_j(T)}{\sum\limits_j \omega_j(T)}&=&\lim\limits_{T\to\infty}\frac{\sum\limits_j \chi_\mu(X_j(p))\omega_j(T)}{\sum\limits_j \omega_j(T)}\\
&=&\lim\limits_{T\to\infty}\frac{\sum\limits_j A_j(p^{\aleph^{-1}(\mu)}) \omega_j(T)}{\sum\limits_j \omega_j(T)}\\
&=&0\\
&=&\int_{T_0/W}f(x)\dd x,
\end{eqnarray*}
where $A_j(p^{\aleph^{-1}(\mu)})$ means $A_j(p^{l_1},...,p^{l_{N-1}})$ if $\aleph((l_1,...,l_{N-1}))=\mu$.
Because $\aleph$ is bijective, we have $\aleph^{-1}(\mu)\neq (0,...,0)$ with $\mu\neq \textbf{0}$
and $$\lim\limits_{T\to\infty}\frac{\sum\limits_j A_j(p^{\aleph^{-1}(\mu)}) \omega_j(T)}{\sum\limits_j \omega_j(T)}=0$$ from Theorem \ref{gk} ($N=3$) and Conjecture \ref{or2} ($N\geq 4$).
\end{proof}

\begin{theorem}[\textbf{Main theorem I}]\label{main1}
Assume Conjecture \ref{or2} if $N\geq 4$.
Assume the Ramanujan conjecture $X_j(p)\in T_0/W$.
For any continuous test function $f\in \mathcal{C}(T_0/W)$, we have the equality
$$\lim\limits_{T\to \infty}\frac{\sum\limits_{j=1}^\infty f(X_j(p))\omega_j(T)}{\sum\limits_{j=1}^\infty \omega_j(T)}=\int_{T_0/W}f(x)\dd x.$$
\end{theorem}

\begin{proof}
We have already proved this theorem when $f\in \mathcal{B}$ and $\mathcal{B}$ is a dense subspace of $\mathcal{C}(T_0/W)$.
We need a little bit of analysis to complete the proof.
For $T\gg 1$, we define a linear functional on $\mathcal{C}(T_0/W)$ by $$\mathbb{L}_T(g)=\frac{\sum\limits_j g(X_j(p))\omega_j(T)}{\sum\limits_j \omega_j(T)}$$ for $g\in\mathcal{C}(T_0/W)$. We define another linear functional by 
$$\mathbb{L}_\infty(g)=\int_{T_0/W} g(x)\dd x$$ for $g\in\mathcal{C}(T_0/W)$.
Both $\mathbb{L}_T$ and $\mathbb{L}_\infty$ are continuous under the supremum norm $\norm{\cdot}_\infty$ and we have the inequalities
$$|\mathbb{L}_T(g)|\leq \norm{g}_\infty\text{ and}\quad|\mathbb{L}_\infty(g)|\leq \norm{g}_\infty.$$
By the Peter-Weyl theorem \ref{peterweyl}, any continuous test function $f$ can be approximated under the topology of the supremum norm by functions in $\mathcal{B}$, i.e.,
we can find a sequence of functions $f_n\in \mathcal{B}$, $n=1,2,...$ such that
$$\lim\limits_{n\to \infty}\norm{f-f_n}_\infty=0.$$
For any $\epsilon>0$, we can find $n'$ such that $\norm{f-f_n}_\infty\leq \frac{\epsilon}{3}$ for any $n>n'$. 
Since we already have $$\lim\limits_{T\to\infty}\mathbb{L}_T(f_{n'+1})=\mathbb{L}_\infty(f_{n'+1})$$ from the previous lemma, we can find $T'$ such that for any $T>T'$ we have
$$\left|\mathbb{L}_T(f_{n'+1})-\mathbb{L}_\infty(f_{n'+1})\right|\leq\frac{\epsilon}{3}.$$
For any $T>T'$,
we have 
\begin{eqnarray*}
\left|\mathbb{L}_T(f)-\mathbb{L}_\infty(f)\right|&\leq&\left|\mathbb{L}_T(f)-\mathbb{L}_T(f_{n'+1})\right|+\left|\mathbb{L}_T(f_{n'+1})-\mathbb{L}_\infty(f_{n'+1})\right|+\left|\mathbb{L}_\infty(f_{n'+1})-\mathbb{L}_\infty(f)\right|\\
&\leq& 2\norm{f-f_{n'+1}}_\infty+\frac{\epsilon}{3}\\
&\leq& \epsilon.
\end{eqnarray*}
It follows that the limit $\lim\limits_{T\to\infty} \mathbb{L}_T(f)$ exists
and equals $\mathbb{L}_\infty(f)$.
\end{proof}

\section{A Long Proof of the Main Theorem Without the Assumption of the Ramanujan Conjecture}

In this section, we are going to prove our main theorem without the assumption of the Ramanujan conjecture. Additionally our main theorem will give insight into the Ramanujan conjecture because it will imply a statistical examination of it.

\begin{lemma}
Denote $A_j(1,...,\underset{(N-k)\tth \text{position}}{p},...,1)$ by $A_j[k]$ for abbreviation.
We have $A_j[k]=\overline{A_j[N-k]}$ for $k=1,2,...,N-1$.
\end{lemma}
\begin{proof}
This is due to unitaricity.  See p. 271 \cite{g}.
\end{proof}

Denote $\left\{
\begin{pmatrix}
\alpha_{1}&&\\
&\ddots&\\
&&\alpha_{N}
\end{pmatrix}:\alpha_i\in\mathbb{C}^*\text{ and }|\alpha_i|\leq p^{\frac{1}{2}}\text{ for all }i, \prod\limits_{i=1}^N \alpha_i=1, 
\right\}$ by $T_1$.
Recall that $T_0$ and $T$ are tori of $\ssu{N}$ and $\ssl{N}{C}$ respectively. We have $T_0\subset T_1\subset T$. We shall note that $T_1$ is a compact set.

\begin{lemma}
The Satake parameter $X_j(p)$ of a Hecke-Maass form $\phi_j$ lies in $T_1/W$.
\end{lemma}
\begin{proof}
See \cite{jacquet}. This is essentially a bound toward the Ramanujan conjecture at $p$.
\end{proof}
We can replace $p^{\frac{1}{2}}$ with $p^{\frac{1}{2}-\frac{1}{N^2+1}}$ in the definition of $T_1$ and the previous lemma still holds from the work of \cite{lrs}. Neither is particularly necessary because any bound which does not depend on $\phi_j$ is good enough for us.
We define an injective map $\varrho: T_1/W\to \mathbb{C}^{N-1}$ by
$$\begin{pmatrix}
\alpha_{1}&&\\
&\ddots&\\
&&\alpha_{N}
\end{pmatrix}\mapsto \left(\chi_1\left(\begin{pmatrix}
\alpha_{1}&&\\
&\ddots&\\
&&\alpha_{N}
\end{pmatrix}\right),...,\chi_{N-1}\left(\begin{pmatrix}
\alpha_{1}&&\\
&\ddots&\\
&&\alpha_{N}
\end{pmatrix}\right)\right).$$
This is a well-defined map because $\chi_k$ is invariant under the action of the Weyl group $W$. Its image $\text{Im }\varrho$ is a compact set in $\mathbb{C}^{N-1}$. This map establishes the equivalence between  $\mathcal{C}(T_1/W)$ and $\mathcal{C}(\text{Im }\varrho)$ the space of continuous functions on $\text{Im }\varrho$. By the Stone-Weierstrass theorem, polynomials in $z_k$ and $\overline{z_k}$ for $k=1,2,...,N-1$ on $\mathbb{C}^{N-1}=\{(z_1,...,z_{N-1}):z_k\in\mathbb{C})\}$ are dense in $\mathcal{C}(\text{Im }\varrho)$.

\begin{lemma}\label{mainlemma}
Assume Conjecture \ref{or2} if $N\geq 4$.
Let $i_k$ and $i'_k$ be non-negative integers for $k=1,2,...,N-1$. We have
$$\lim\limits_{T\to\infty}\frac{\sum\limits_j \prod\limits_{k=1}^{N-1}A_j[k]^{i_k}{\overline{ A_j[k]}}^{i'_k}\omega_j(T)}{\sum\limits_j \omega_j(T)}=\int_{T_0/W} \prod\limits_{k=1}^{N-1}\chi_k(x)^{i_k}\overline{\chi_k(x)}^{i'_k} \dd x.$$
\end{lemma}

\begin{proof}
By Proposition \ref{CS}, we have $A_j[k]=\chi_k(X_j(p))$. 
The character of the tensor product representation $\bigotimes\limits_{k=1}^{N-1} \left(V_k^{\otimes i_k}\otimes V_{N-k}^{\otimes i'_k}\right)$ is $\prod\limits_{k=1}^{N-1}\chi_k^{i_k}\chi_{N-k}^{i'_k}$. By the basic Lie theory, any finite-dimensional representation is a direct sum of irreducible representations and we have
$$\bigotimes\limits_{k=1}^{N-1} \left(V_k^{\otimes i_k}\otimes V_k^{\otimes i'_k}\right)=\bigoplus\limits_{\mu\in \Lambda\cap C}V_\mu^{\oplus a_\mu},$$ where $a_\mu$ is the multiplicity of $V_\mu$.
Hence we have the corresponding identity of characters 
$\prod\limits_{k=1}^{N-1}\chi_k^{i_k}\chi_{N-k}^{i'_k}=\sum\limits_{\mu\in \Lambda\cap C} a_\mu \chi_\mu$. We have 

\begin{eqnarray*}
\lim\limits_{T\to\infty}\frac{\sum\limits_j \prod\limits_{k=1}^{N-1}A_j[k]^{i_k}{\overline{ A_j[k]}}^{i'_k}\omega_j(T)}{\sum\limits_j \omega_j(T)}
&=&\lim\limits_{T\to\infty}\frac{\sum\limits_j \prod\limits_{k=1}^{N-1}A_j[k]^{i_k}{A_j[N-k]}^{i'_k}\omega_j(T)}{\sum\limits_j \omega_j(T)}\\
&=&\lim\limits_{T\to\infty}\frac{\sum\limits_j \prod\limits_{k=1}^{N-1}\chi_k(X_j(p))^{i_k}{\chi_{N-k}(X_j(p))}^{i'_k}\omega_j(T)}{\sum\limits_j \omega_j(T)}\\
&=&\lim\limits_{T\to\infty}\frac{\sum\limits_j \sum\limits_\mu a_\mu\chi_\mu(X_j(p))\omega_j(T)}{\sum\limits_j \omega_j(T)}\\
&=&\sum\limits_\mu a_\mu\lim\limits_{T\to\infty}\frac{\sum\limits_j \chi_\mu(X_j(p))\omega_j(T)}{\sum\limits_j \omega_j(T)}\\
&=&a_\textbf{0}.
\end{eqnarray*}
On the other side, we have
\begin{eqnarray*}
\int_{T_0/W} \prod\limits_{k=1}^{N-1}\chi_k(x)^{i_k}\overline{\chi_k(x)}^{i'_k} \dd x
&=&\int_{T_0/W} \prod\limits_{k=1}^{N-1}\chi_k(x)^{i_k}\overline{\chi_k(x)}^{i'_k} \dd x\\
&=&\int_{T_0/W} \sum\limits_\mu a_\mu\chi_\mu(x) \dd x\\
&=&\sum\limits_\mu a_\mu \int_{T_0/W}  \chi_\mu(x) \dd x\\
&=&a_{\textbf{0}}.
\end{eqnarray*}
Hence we establish the identity.
\end{proof}

\begin{theorem}[\textbf{Main theorem II}]\label{main2}
Assume Conjecture \ref{or2} if $N\geq 4$.
For any continuous test function $f: T/W\to \mathbb{C}$ we have the equality
\begin{equation}\label{main2eq}
\lim\limits_{T\to \infty}\frac{\sum\limits_{j=1}^\infty f(X_j(p))\omega_j(T)}{\sum\limits_{j=1}^\infty \omega_j(T)}=\int_{T_0/W}f(x)\dd x.
\end{equation}
\end{theorem}

\begin{proof}
The composition $f\circ\varrho^{-1}$ is a continuous function in $\mathcal{C}(\text{Im }\varrho)$. 
We only need to prove that for any continuous function $F:\text{Im }\varrho\to\mathbb{C}$, we have 
\begin{equation}\label{main2prof}
\lim\limits_{T\to \infty}\frac{\sum\limits_j F(A_j[1],...,A_j[N-1])\omega_j(T)}{\sum\limits_j \omega_j(T)}=\int_{T_0/W}(F\circ\varrho)(x)\dd x.
\end{equation}
By the previous lemma, we have proved Equation \ref{main2prof} for $F$ being any monomial $(z_1,...,z_{N-1})\mapsto\prod\limits_{k=1}^{N-1} z_k^{i_k}\overline{z_k}^{i'_k}$. By linear combination, we prove Equation \ref{main2prof} for 
all polynomials in $z_1$, $\overline{z_1}$,..., $z_{N-1}$, $\overline{z_{N-1}}$ on $\mathbb{C}^{N-1}$. 
By the Stone-Weierstrass theorem, such polynomials are dense in $\mathcal{C}(\text{Im }\varrho)$ under the topology of supremum norm. Apply the same epsilon-delta argument in the proof of Theorem \ref{main1} and we complete the proof.
\end{proof}
\begin{remark}
Theorem \ref{main2} essentially proves the Ramanujan conjecture in average with respect to varying $\phi_j$ and fixed $p$. On the left side of Equation \ref{main2eq}, $X_j(p)$ is only known to lie in $T_1/W$, while on the right side of the same equation, the Sato-Tate measure is supported on exactly $T_0/W$.
\end{remark}

\section{Case of N=3 and the Rate of Convergence}
Fix $N=3$ in this section and $\omega_j(T)$ is as defined in Theorem \ref{gk}.
By an application of the error terms obtained by \cite{blomer} in Theorem \ref{gk}, we will prove an effective version of Lemma \ref{mainlemma} for $N=3$, which estimates the rate of convergence of the limit in that lemma.

Let us recall the formal characters of a representation in the special case of the root system of type $A_2$. Characters of $\ssu{3}$ or $\ssl{3}{C}$ are generated by $e^{\aleph((0,1))}=e^{\epsilon_1-\frac{\epsilon_1+\epsilon_2+\epsilon_3}{3}}$
and
$e^{\aleph((1,0))}=e^{-\epsilon_3+\frac{\epsilon_1+\epsilon_2+\epsilon_3}{3}}$ over $\mathbb{Z}$ as rational functions.
We have $$\chi_1=e^{\epsilon_1-\frac{\epsilon_1+\epsilon_2+\epsilon_3}{3}}+e^{\epsilon_2-\frac{\epsilon_1+\epsilon_2+\epsilon_3}{3}}+e^{\epsilon_3-\frac{\epsilon_1+\epsilon_2+\epsilon_3}{3}}
\;\text{ and }\;
\chi_2=e^{-\epsilon_1+\frac{\epsilon_1+\epsilon_2+\epsilon_3}{3}}+e^{-\epsilon_2+\frac{\epsilon_1+\epsilon_2+\epsilon_3}{3}}+e^{-\epsilon_3+\frac{\epsilon_1+\epsilon_2+\epsilon_3}{3}}.
$$
for the two fundamental representation $V_1$ and $V_2=\wedge^2 V_1$.

\begin{theorem}[\textbf{Rate of convergence for N=3}] Let us fix $\epsilon>0$ and keep $\omega_j(T)$ as defined in Theorem \ref{gk}. 
Let $\theta\leq 7/64$ be a bound toward the Ramanujan conjecture on GL(2).
Let $f\circ \varrho$ be a monomial $(z_1,z_2)\mapsto z_1^{i_1}\overline{z_1}^{i'_1}z_2^{i_2}\overline{z_2}^{i'_2}$ for non-negative integers $i_1$, $i'_1$, $i_2$, $i'_2$. We have 
$$\frac{\sum\limits_{j=1}^\infty f(X_j(p))\omega_j(T)}{\sum\limits_{j=1}^\infty \omega_j(T)}=\int_{T_0/W}f(x)\dd x+ \error
$$ as $T\gg 1$ and $P=p^{i_1+i'_1+i_2+i'_2}$.
\end{theorem}

\begin{proof}
Let us recall $\Omega_3=\{(l_1,l_2):l_1,l_2\in\mathbb{Z},l_1\geq  0,l_2\geq 0\}$
and the map $\aleph:\Omega_3\to \Lambda\cap C$ defined by
$(l_1,l_2)\mapsto l_2(\epsilon_1-\frac{\epsilon_1+\epsilon_2+\epsilon_3}{3})+l_1(\frac{\epsilon_1+\epsilon_2+\epsilon_3}{3}-\epsilon_3)$.
We also recall  $V_1^{\otimes (i_1+i'_2)}\otimes V_2^{\otimes (i_2+i'_1)}=\bigoplus\limits_{\mu}V_\mu^{a_\mu}$ 
and $\chi_1^{i_1+i'_2}\chi_2^{i_2+i'_1}=\sum\limits_\mu a_\mu \chi_\mu$ in the proof of Lemma \ref{mainlemma}. 
We have

\begin{eqnarray*}
\frac{\sum\limits_j f(X_j(p))\omega_j(T)}{\sum\limits_j \omega_j(T)}-\int_{T_0/W}f(x)\dd x
&=&-a_\textbf{0}+ \frac{\sum\limits_j \prod\limits_{k=1}^{2}A_j[k]^{i_k}{\overline{ A_j[k]}}^{i'_k}\omega_j(T)}{\sum\limits_j \omega_j(T)}\\
&=& -a_\textbf{0}+ \frac{\sum\limits_j \prod\limits_{k=1}^{2}A_j[k]^{i_k}{A_j[3-k]}^{i'_k}\omega_j(T)}{\sum\limits_j \omega_j(T)}\\
&=& -a_\textbf{0}+ \frac{\sum\limits_j \chi_1^{i_1+i'_2}(X_j(p))\chi_2^{i_2+i'_1}(X_j(p))\omega_j(T)}{\sum\limits_j \omega_j(T)}\\
&=&-a_\textbf{0}+\frac{\sum\limits_j \sum\limits_\mu a_\mu\chi_\mu(X_j(p))\omega_j(T)}{\sum\limits_j \omega_j(T)}\\
&=&-a_\textbf{0}+\sum\limits_{\mu\in \Lambda\cap C} a_\mu \frac{\sum\limits_j  \chi_\mu(X_j(p))\omega_j(T)}{\sum\limits_j \omega_j(T)}\\
&=&-a_\textbf{0}+\sum\limits_{l_1\geq 0, l_2\geq 0} a_{\aleph((l_1,l_2))}\frac{\sum\limits_j  A_j(p^{l_1},p^{l_2})\omega_j(T)}{\sum\limits_j \omega_j(T)}\\
&=&\sum\limits_{l_1\geq 0, l_2\geq 0} a_{\aleph((l_1,l_2))}\left(\frac{\sum\limits_j  A_j(p^{l_1},p^{l_2})\omega_j(T)}{\sum\limits_j \omega_j(T)}-\delta_{l_1,0}\delta_{l_2,0}\right).
\end{eqnarray*}
From the error terms in Theorem \ref{gk},
we know that 
$({\sum\limits_j  A_j(p^{l_1},p^{l_2})\omega_j(T)}/{\sum\limits_j \omega_j(T)}-\delta_{l_1,0}\delta_{l_2,0})$
is up to $\mathcal{O}\left((T^2{(p^{l_1+l_2})}^{1/2}+T^3{(p^{l_1+l_2})}^\theta+{(p^{l_1+l_2})}^{5/3})T^{-5+\epsilon}{(p^{l_1+l_2})}^\epsilon\right)$. 
We need a good bound for $\sum\limits_{l_1\geq 0,l_2\geq 0}a_{\aleph((l_1,l_2))}p^{\alpha(l_1+l_2)}$ with $\alpha=\theta+\epsilon, 1/2+\epsilon, 5/3+\epsilon$.
Recall that $a_\mu$ is the multiplicity of $V_\mu$ in the decomposition of the representation
$$V_1^{\otimes(i_1+i'_2)}\otimes V_2^{\otimes(i'_1+i_2)}=\bigoplus_{\mu\in \Lambda\cap C} V_\mu^{a_\mu}.$$
Because of $$\text{dim}\{v\in V_\mu: t.v=\mu(t)v \text{ for all }t\in T_0\}=1$$
we have
\begin{eqnarray*}
a_\mu&\leq& \text{dim}\{v\in V_1^{\otimes(i_1+i'_2)}\otimes V_2^{\otimes(i'_1+i_2)} : t.v=\mu(t)v\text{ for all }t\in T_0\}\\
&=&\chi_1^{i_1+i'_2}\chi_2^{i_2+i'_1}\bigg|_{e^\mu},
\end{eqnarray*} 
where $\bigg|_{e^\mu}$
means taking the coefficient before $e^\mu$.
Hence we obtain

\begin{eqnarray*}
\sum\limits_{l_1\geq 0,l_2\geq 0}a_{\aleph((l_1,l_2))}p^{\alpha(l_1+l_2)}
&\leq& \sum\limits_{l_1\geq 0,l_2\geq 0} \left(\chi_1^{i_1+i'_2}\chi_2^{i_2+i'_1}\bigg|_{e^{\aleph((l_1,l_2))}}\right)p^{\alpha(l_1+l_2)}\\
&=&\sum\limits_{\mu\in \Lambda\cap C} \left(\chi_1^{i_1+i'_2}\chi_2^{i_2+i'_1}\bigg|_{e^\mu}\right)e^\mu \Bigg|\Bigg|_{e^{\aleph((1,0))}=e^{\aleph((0,1))}=p^\alpha}\\
&<&\sum\limits_{\mu\in\Lambda} \left(\chi_1^{i_1+i'_2}\chi_2^{i_2+i'_1}\bigg|_{e^\mu}\right)e^\mu \Bigg|\Bigg|_{e^{\aleph((1,0))}=e^{\aleph((0,1))}=p^\alpha}\\
&=&\chi_1^{i_1+i'_2}\chi_2^{i_2+i'_1} \Bigg|\Bigg|_{e^{\aleph((1,0))}=e^{\aleph((0,1))}=p^\alpha}\\
&=&(p^\alpha+1+p^{-\alpha})^{i_1+i'_1+i_2+i'_2}\\
&=&\mathcal{O}\left(p^{\alpha(i_1+i'_1+i_2+i'_2)}\right),
\end{eqnarray*}
where $\Bigg|\Bigg|_{e^{\aleph((1,0))}=e^{\aleph((0,1))}=p^\alpha}$ means replacing $e^{\aleph((1,0))}$ and $e^{\aleph((0,1))}$ with $p^\alpha$ in $\chi_1^{i_1+i'_2}\chi_2^{i_2+i'_1}$, which is a rational function generated by $e^{\aleph((1,0))}$ and $e^{\aleph((0,1))}$.
In continuation, we have
\begin{eqnarray*}
\frac{\sum\limits_j f(X_j(p))\omega_j(T)}{\sum\limits_j \omega_j(T)}&-&\int_{T_0/W}f(x)\dd x\\
&=&\sum\limits_{l_1\geq 0, l_2\geq 0} a_{\aleph((l_1,l_2))}\left(\frac{\sum\limits_j  A_j(p^{l_1},p^{l_2})\omega_j(T)}{\sum\limits_j \omega_j(T)}-\delta_{l_1,0}\delta_{l_2,0}\right)\\
&=&\sum\limits_{l_1\geq 0, l_2\geq 0} a_{\aleph((l_1,l_2))}\mathcal{O}\left((T^2{(p^{l_1+l_2})}^{1/2}+T^3{(p^{l_1+l_2})}^\theta+{(p^{l_1+l_2})}^{5/3})T^{-5+\epsilon}{(p^{l_1+l_2})}^\epsilon\right)\\
&=&\mathcal{O}\left((T^2P^{1/2}+T^3P^\theta+P^{5/3})T^{-5+\epsilon}P^\epsilon\right),
\end{eqnarray*}
with $P=p^{i_1+i'_1+i_2+i'_2}$.
Hence we complete the proof.
\end{proof}

\begin{small}
\section*{Acknowledgment}
The main theorem here appeared firstly at the author's PhD thesis, \cite{zhou}. 
This paper presents a more sophisticated formulation and a different proof.
The author is grateful to his thesis advisor, Dorian Goldfeld, who brought the author to this topic of research, gave the author much guidance, and painstakingly read the manuscript.
The author would like to thank Peter Sarnak for many valuable comments,
in particular, suggesting using Satake parameters instead of Hecke eigenvalues for equidistribution theorems in higher dimensions.
The author would like to thank the anonymous referee who made valuable comments.
The author would like to thank Hang Xue for helpful discussions.
\end{small}

\end{document}